\documentclass{amsart}

\usepackage{winart3}

\pgfplotsset{compat=1.18}

\usetikzlibrary{calc}

\tikzset{description/.style={fill=white,midway,font=\small}}

\title{The Hochschild cohomology ring of monomial algebras}

\date{\today}


\author[D.~Artenstein]{Dalia Artenstein}
\address{Dalia~Artenstein, Institute of Mathematics and Stadistics Rafael Laguardia, Faculty of Engineering, University of the Republic, Av. Julio Herrera y Reissig 565, Montevideo, Uruguay}
\email{darten@fing.edu.uy}

\author[J.~C.~Letz]{Janina C. Letz}
\address{Janina~C.~Letz,
Faculty of Mathematics,
Bielefeld University,
PO Box 100 131,
33501 Bielefeld,
Germany \newline
UCLA Department of Mathematics,
PO Box 951555, 
Los Angeles, CA 90095, 
USA}
\email{jletz@math.uni-bielefeld.de}

\author[A.~Oswald]{Amrei Oswald}
\address{Amrei~Oswald, Department of Mathematics, University of Washington, Seattle, WA 98195, USA}
\email{amreio@uw.edu}

\author[A.~Solotar]{Andrea Solotar}
\address{Andrea~Solotar, Departamento de Matem\'atica, 
 Facultad de Ciencias Exactas y Naturales, Universidad de Buenos Aires, 
 and IMAS-CONICET,
 Pabellon I -- Ciudad Universitaria,
 Buenos Aires, 1428\\ 
 Argentina}
\email{asolotar@dm.uba.ar}

\begin{document}

\begin{abstract}
We give an explicit description of a diagonal map on the Bardzell resolution for any monomial algebra, and we use this diagonal map to describe the cup product on Hochschild cohomology. Then, we prove that the cup product is zero in positive degrees for triangular monomial algebras. Our proof uses the graded-commutativity of the cup product on Hochschild cohomology and does not rely on explicit computation of the Hochschild cohomology modules.
\end{abstract}

\maketitle

\tableofcontents


\section{Introduction}

Given a field $\kk$ and an associative $\kk$-algebra $A$, the Hochschild cohomology of $A$ with coefficients in the $A$-bimodule $A$ is the graded vector space 
\[
\HH A = \bigoplus_{i \geqslant 0}\Ext{i}{A \otimes \op{A}}{A}{A}\,.
\]
The Hochschild cohomology of an associative algebra has a rich structure, which provides several derived invariants that are useful when studying the representations of the given algebra. The cup product, $\smile:\HH A \otimes \HH A \to \HH A$, gives $\HH A$ the structure of a graded-commutative $\kk$-algebra, as has been proved in several different ways. The proof of Murray Gerstenhaber \cite{Gerstenhaber:1963} gave rise to the definition of the Gerstenhaber bracket.
Both the cup product and the Gerstenhaber bracket were originally defined in terms of the bar resolution, which is usually very inefficient for explicit computations of the Hochschild cohomology. It has been known for a long time that the Hochschild cohomology can be computed by any resolution of $A$ as an $A$-bimodule; see \cite[Chapter~IX, \S6]{Cartan/Eilenberg:1956}. However, that the cup product and Gerstenhaber bracket can be computed using any resolution are more recent results; see \cite{Siegel/Witherspoon:1999} and \cite{Negron/Witherspoon:2016,Volkov:2019} respectively. In fact, the cup product coincides with the Yoneda product.

It is possible to avoid using the bar resolution for the computation of the cup product and use any resolution $P$---perhaps minimal---whenever it is equipped with a diagonal map $\Delta \colon P \to P \otimes_A P$; see for example \cite{Witherspoon:2019}. We use this strategy to compute the cup product for triangular monomial algebras and prove the following theorem.

\begin{introthm}[\cref{main}] \label{intromain}
Consider a finite triangular monomial algebra, $A$. The cup product on $\HH{A}$ is zero in positive degrees.
\end{introthm}

There are many classes of algebras for which the cup product is zero in positive degrees. 
In \cite[Lemma~3.1]{Cibils:1998}, Claude Cibils showed this for a finite dimensional path algebra $A=\kk Q/I$ with radical square zero and a finite quiver $Q$ that is not a \emph{crown}. In \emph{loc.~cit.} cycles are called crowns. As described in \cite{Green/Solberg:2002}, there exist self-injective algebras which have nonzero cup product in positive degrees. Moreover, in \cite{Green/Solberg:2002}, Edward Green and {\O}yvind Solberg construct many algebras that are not self-injective for which the cup product is nonzero in positive degrees. Next, Juan Carlos Bustamante \cite{Bustamante:2006} proved that the cup product is zero in positive degrees for triangular quadratic string algebras and conjectured that this holds for any triangular monomial algebra as in \cref{intromain}. In \cite[Theorem~5.3]{Redondo/Roman:2014}, Mar{\'\i}a Julia Redondo and Lucrecia Rom{\'a}n verified this conjecture for triangular string algebras.

For the description of the Hochschild cohomology, we use the Bardzell resolution introduced in \cite{Bardzell:1997}. This is a minimal resolution of $A$ as an $A$-bimodule, hence it is much more efficient than the bar resolution. We give an explicit description of a diagonal map on the Bardzell resolution for any monomial algebra in \cref{diagonal}. We then use this diagonal map and the fact that the cup product is graded-commutative to prove \cref{intromain} in \cref{sec:cup-product}. Our proof does not rely on an explicit computation of the Hochschild cohomology modules. 

\begin{ack}
We started working on this article during WINART3, and we are grateful to the organizers of that meeting. We would also like to thank Sibylle Schroll for her help and suggestions. Letz was partly supported by the Deutsche Forschungsgemeinschaft (SFB-TRR 358/1 2023 - 491392403) and by the Alexander von Humboldt Foundation in the framework of a Feodor Lynen research fellowship endowed by the German Federal Ministry of Education and Research. Oswald was supported by a grant from the Simons Foundation Targeted Grant (917524) to the Pacific Institute for the Mathematical Sciences. Artenstein and Solotar were partially supported by the project Mathamsud-AREPTHEO. Solotar was partially supported by PIP-CONICET 11220200101855CO. She also thanks Guangdong Technion Israel Institute of Technology, where part of this work has been done.
\end{ack} 

\section{Ambiguities}

Let $\kk$ be a field and $A = \kk Q/I$ a finite dimensional monomial algebra. We denote by $\cB$ a basis of paths of $A$; such a basis is unique since $A$ is monomial. In particular, a path is zero in $A$ if and only if it is not contained in $\cB$. 

\begin{notation}
A finite quiver $Q = (Q_0,Q_1,s,t)$ consists of
\begin{enumerate}
\item a finite set of \emph{vertices} $Q_0$, 
\item a finite set of \emph{arrows} $Q_1$, 
\item a source map $s \colon Q_1 \to Q_0$, and
\item a target map $t \colon Q_1 \to Q_0$. 
\end{enumerate}
A \emph{path} $p$ is a sequence of arrows $\alpha_n \ldots \alpha_1$ where $t(\alpha_i) = s(\alpha_{i+1})$. Graphically the path can be depicted as
\begin{equation*}
\begin{tikzcd}
\bullet \ar[r,"\alpha_1"] \& \bullet \ar[r,"{\alpha_2}"] \& \ldots \ar[r,"{\alpha_n}"] \& \bullet \nospacepunct{.}
\end{tikzcd}
\end{equation*}
We set $s(p) \colonequals s(\alpha_1)$ the \emph{source of $p$} and $t(p) \colonequals t(\alpha_n)$ the \emph{target of $p$}. We say the \emph{length} of the path $p = \alpha_n \ldots \alpha_1$ is $n$. We identify a path of length zero with its vertex. We write $qp$ for the concatenation of the paths $p$ and $q$ if $t(p) = s(q)$. 

Let $p$ and $q$ be paths such that $p = b q a$. We say $q$ is a \emph{divisor} of $p$. In general the position of $q$ in $p$, that is the paths $a$ and $b$, need not be unique. We write $q \leq p$ to mean that $q$ is a divisor of $p$ and the position of $q$ in $p$ is fixed. In this situation we set
\begin{equation*}
\pre_p(q) \colonequals a \quad \text{and} \quad \suf_p(q) \colonequals b
\end{equation*}
for the \emph{prefix} and \emph{suffix of $q$ in $p$}, respectively. We drop the subscript when it is clear from the context. If $b = 1$, we say $q$ is a \emph{suffix} of $p$, and if $a = 1$, we say $q$ is a \emph{prefix} of $p$. If $q$ is a divisor/suffix/prefix of $p$ and $p \neq q$, we say $q$ is a \emph{proper} divisor/suffix/prefix, and we write $q \lneq p$. 
\end{notation}

\begin{definition} \label{ambiguity}
Let $n \geq -1$ be an integer. A path in $Q$ is a \emph{left $n$-ambiguity} if it decomposes as $p = u_{-1} \ldots u_n$ such that
\begin{enumerate}
\item \label{ambiguity:pieces} $u_{-1} \in Q_0$, $u_0 \in Q_1$, $u_i \in \cB$, and
\item \label{ambiguity:minimal} for any $0 \leq i \leq n-1$ we have $u_i u_{i+1} \in I$ and any proper suffix of $u_i u_{i+1}$ does not lie in $I$. 
\end{enumerate}
In practice, we drop the trivial path $u_{-1}$ and write $u_0 \ldots u_n$ for the decomposition of a left ambiguity. 

Dually, a path $p$ is a \emph{right $n$-ambiguity} if it decomposes as $p = v_n \ldots v_{-1}$ such that the above properties hold when we replace suffix by prefix.
\end{definition}

Left ambiguities are called \emph{right $n$-chains} in \cite[Section~3]{Skoeldberg:2008}. Our use of the term ambiguity follows \cite[Definition~3.1]{Chouhy/Solotar:2015} and alludes to the fact that left and right ambiguities are equivalent; see \cite[Lemma~3.1]{Bardzell:1997} and also \cite[Lemma~1]{Skoeldberg:2008}. Hence we will simply say $n$-ambiguity instead of left or right $n$-ambiguity. Moreover, given an $n$-ambiguity $p$ and decompositions as left and right ambiguities $p = u_0 \ldots u_n = v_n \ldots v_0$, respectively, we have
\begin{equation} \label{compare_left_right}
v_i \leq u_{n-i} u_{n-i+1} \quad \text{and} \quad u_i \leq v_{n-i+1} v_{n-i}
\end{equation}
for any $1 \leq i \leq n$. 

We denote the set of $n$-ambiguities by $\Gamma_n$. This set is denoted by $AP(n+1)$ in \cite{Bardzell:1997} and by $AP_{n+1}$ in \cite{Redondo/Roman:2018a,Redondo/Roman:2018b}. 

For small $n$, we have the following descriptions of the sets of ambiguities $\Gamma_n$.
\begin{enumerate}
\item The $(-1)$-ambiguities are precisely the trivial paths of length zero, that is $\Gamma_{-1} = Q_0$. 
\item The $0$-ambiguities are precisely the arrows, that is $\Gamma_0 = Q_1$. 
\item The $1$-ambiguities are the minimal generating set of paths of $I$. 
\end{enumerate}

For quadratic algebras the ambiguities have particularly nice decompositions:

\begin{example}
Let $A = \kk Q/I$ be a finite dimensional quadratic monomial algebra. Then $I$ is generated by paths of length 2. For any left ambiguity $p = u_0 \ldots u_n$ each $u_i$ is an arrow in $Q$ and $u_i u_{i+1}$ lies in the minimal generating set of $I$. In particular, setting $v_i \colonequals u_{n-i}$ yields a decomposition of $p$ as a right ambiguity. Moreover, for any $0 \leq i \leq n-1$ the path $u_i u_{i+1}$ is a 1-ambiguity.
\end{example}

The latter property needs not hold when $I$ is generated by paths of length more than 2. In this case, $u_i u_{i+1}$ can have a proper \emph{prefix} that lies in $I$. 

\begin{example} \label{counterex-2-amb-I}
We consider the quiver 
\begin{equation*}
Q = 
\begin{tikzcd}
\bullet \ar[d,"\alpha" swap] \& \bullet \ar[l,"\delta" swap] \\
\bullet \ar[r,"\beta" swap] \& \bullet \ar[u,"\gamma" swap]
\end{tikzcd}
\quad \text{with}\quad I = (\gamma \beta \alpha,\alpha \delta \gamma)\,.
\end{equation*}
Then, $\alpha | \delta \gamma | \beta \alpha$ is a left 2-ambiguity. The notation indicates that $u_0 = \alpha$, $u_1 = \delta \gamma$ and $u_2 = \beta \alpha$. In this left 2-ambiguity, we have $u_1u_2 = \delta \gamma \beta \alpha \gneq \gamma \beta \alpha \in I$. 
\end{example}

The following \namecref{amb_containment} utilizes that $u_{i+1}$ is chosen minimally so that $u_i u_{i+1}$ lies in $I$. It will be used repeatedly in the sequel.

\begin{lemma} \label{amb_containment}
Let $m \geq n$ be integers and let 
\begin{equation*}
p = u_0 \ldots u_m \geq u'_0 \ldots u'_n = q
\end{equation*}
be decompositions as left $m$- and $n$-ambiguities, respectively. We assume $\ell \leq m$ is maximal such that $u_\ell \ldots u_m \geq q$. Then
\begin{equation} \label{amb_containment:comparison}
\begin{aligned}
u_{\ell+i} \geq u'_i & \quad \text{when $i$ even} \\
u_{\ell+i} \leq u'_i & \quad \text{when $i$ odd}
\end{aligned} \quad \text{for} \quad 0 \leq i \leq n\,.
\end{equation}
In particular, this yields $m \geq l+n$ and
\begin{enumerate}
\item if $n$ is even, then $u_\ell \ldots u_{\ell+n} \geq q$ and $u_\ell \ldots u_{\ell+n-1} \not\geq q$.
\item if $n$ is odd, then $u_{\ell+1} \ldots u_{\ell+n} \leq q$. 
\end{enumerate}
Moreover, if $n$ is odd and $q$ a suffix of $u_\ell \ldots u_m$, then $q = u_\ell \ldots u_{\ell+n}$. 

An analogous statement holds for right ambiguities.
\end{lemma}
\begin{proof}
We show \cref{amb_containment:comparison} by induction on $i$. By assumption the claim holds for $i=0$. We assume $u_{\ell+i} \geq u'_i$. Since no proper suffix of $u_{\ell+i} u_{\ell+i+1}$ lies in $I$ and $u'_i u'_{i+1} \in I$, we have $u_{\ell+i+1} \leq u'_{i+1}$. A similar argument yields $u_{\ell+i+1} \geq u'_{i+1}$ when {$u_{\ell+i} \leq u'_i$}. 

Now let $n$ be odd and $q$ a suffix of $u_\ell \ldots u_m$. We show that, in this case, we have $u_\ell \ldots u_{\ell+i} = u'_0 \ldots u'_i$ when $i$ is odd in the induction above. For $i=-1$ there is nothing to show. We assume it holds for $i$. Then $u'_{i+1} u'_{i+2}$ and $u_{\ell+i+1} u_{\ell+i+2}$ have the same source, so one is the suffix of the other. Since both lie in $I$ and neither has a proper suffix in $I$, they are equal. Taking $i=n$ yields the desired statement.
\end{proof}

\begin{corollary} \label{amb_uniqueness}
Let $p$ be an $n$-ambiguity.
\begin{enumerate}
\item No proper divisor of $p$ is an $n$-ambiguity.
\item The decompositions of $p$ as a left and right ambiguity are unique. \qed
\end{enumerate}
\end{corollary}

Let $p$ be a left $n$-ambiguity with decomposition $p = u_0 \ldots u_n$. The suffix
\begin{equation*}
\ambsuf{m}{p} \colonequals u_0 \ldots u_m
\end{equation*}
is, with this decomposition, a left $m$-ambiguity for every $-1 \leq m \leq n$. 

Analogously, if $p$ is a right $n$-ambiguity with decomposition $p = v_n \ldots v_0$ we set
\begin{equation*}
\ambpre{m}{p} \colonequals v_m \ldots v_0\,.
\end{equation*}

By \cref{compare_left_right}, any $n$-ambiguity $p$ has a decomposition
\begin{equation} \label{amb_split_suf_pre}
p = \ambsuf{j}{p} b \ambpre{i}{p}
\end{equation}
for $i+j=n-1$ and some $b \in B$ with $b < u_{j+1}$ and $b < v_{i+1}$. In general the remainder $b$ in \cref{amb_split_suf_pre} need not be trivial, as seen in the following example.

\begin{example}
In \cref{counterex-2-amb-I}, the 3-ambiguity $p = \gamma \beta \alpha \delta \gamma \beta \alpha$ has the decompositions as left and right ambiguity
\begin{equation*}
\gamma | \beta \alpha | \delta \gamma | \beta \alpha \quad \text{and}\quad \gamma \beta | \alpha \delta | \gamma \beta | \alpha\,,
\end{equation*}
respectively. Then, we have $\ambsuf{1}{p} = \gamma \beta \alpha$ and $\ambpre{1}{p} = \gamma \delta \alpha$, and the remainder is $b = \alpha$. In this example, for any $n$-ambiguity and $i+j = n-1$, the remainder is a path of length 1. 
\end{example}

\section{Bardzell's resolution}

Let $A = \kk Q/I$ be a monomial algebra and $\Gamma_n$ the set of $n$-ambiguities. We set $E \colonequals \kk Q_0 \cong \kk \Gamma_{-1}$ and $\env{A} \colonequals A \otimes_k \op{A}$. In this section we describe a minimal resolution of $A$ over $\env{A}$ due to Bardzell \cite{Bardzell:1997} that uses $n$-ambiguities as a basis in homological degree $n+1$. Since this resolution was first introduced, the notation has evolved and many arguments in the original paper can be made more concrete. In the sequel, we make use of some of these technical results. For this reason and the convenience of the reader, we will give many of the details. 

We start with a description of the $(n-1)$-ambiguities contained in an $n$-ambiguity. This depends on the parity of $n$. 

\begin{notation}
Let $p$ be an $n$-ambiguity, i.e. $p\in \Gamma_n$. We define
\begin{equation*}
\Sub(p) \colonequals \set{q \in \Gamma_{n-1}}{q \leq p}\,.
\end{equation*}
For every $q \in \Sub(p)$ we have $p = \suf(q) q \pre(q)$. Clearly $\suf(q), \pre(q) \notin I$. By \cref{amb_uniqueness}, we have two distinct elements
\begin{equation*}
\left\{\ambsuf{n-1}{p},\ambpre{n-1}{p}\right\} \subseteq \Sub(p)\,.
\end{equation*}
\end{notation}

We will see that the description of $\Sub(p)$ depends on the parity of $n$. But first we need a technical \namecref{amb_odd_overlap_lift}.

\begin{lemma} \label{amb_odd_overlap_lift}
Let $n$ be an odd integer and $q_1 \neq q_2$ be $n$-ambiguities with {$q_1 a = b q_2$}. If $a \notin I$ and $b \notin I$, then there exist $(n+1)$-ambiguities $p_1$ and $p_2$ such that $p_1,p_2 \leq q_1 a = b q_2$ and $\ambsuf{n}{p_1} = q_1$ and $\ambpre{n}{p_2} = q_2$.
\end{lemma}
\begin{proof}
We assume $a \notin I$ and $b \notin I$. Let 
\begin{equation*}
q_1 = u_0 \ldots u_n \quad \text{and}\quad q_2 = u'_0 \ldots u'_n
\end{equation*}
be decompositions as left $n$-ambiguities. This setup is illustrated in \cref{amb_odd_overlap_lift:proof}.

\begin{figure}[h]
\begin{equation*}
\begin{tikzcd}[row sep=tiny,column sep=small]
\bullet \ar[rr,"a"] \&\& \bullet \ar[rrrr,"{u_n}"] \&\&\&\& \bullet \& \ldots \& \bullet \ar[rrr,"u_1"] \&\&\& \bullet \ar[r,"u_0"] \& \bullet \\
\bullet \ar[rrr,"{u'_n}"] \&\&\& \bullet \ar[rr,"{u'_{n-1}}"] \&\& \bullet \&\& \ldots \&\& \bullet \ar[r,"{u'_0}"] \& \bullet \ar[rr,"b"] \&\& \bullet
\end{tikzcd}
\end{equation*}
\caption{Graphical depiction of the setup of \cref{amb_odd_overlap_lift}.} \label{amb_odd_overlap_lift:proof}
\end{figure}
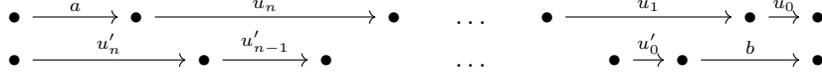

 Since $b \notin I$ and $q_1 \neq q_2$, we have $u_1 \geq u'_0$. By \cref{amb_containment}, $u_n \geq u'_{n-1}$, and so, $u_n a \geq u'_{n-1} u'_n$. We can choose $u_{n+1} \leq a$ minimally such that $u_{n+1} u_n \in I$. Since $a \notin I$, we have $u_{n+1} \notin I$ and $p_1 \colonequals u_0 \ldots u_{n+1} \leq q_1 a$ a $(n+1)$-ambiguity. 

The $(n+1)$-ambiguity $p_2$ exists by the same argument for decompositions of right ambiguities.
\end{proof}

Now we show \cite[Lemma~3.3]{Bardzell:1997}.

\begin{lemma} \label{sub_amb_even}
Let $n$ be an even integer and $p$ an $n$-ambiguity. Then
\begin{equation*}
\Sub(p) = \left\{\ambsuf{n-1}{p},\ambpre{n-1}{p}\right\}\,.
\end{equation*}
\end{lemma}
\begin{proof}
Let $q \in \Sub(p)$ with $q \neq \ambsuf{n-1}{p}$. We can write
\begin{equation*}
\ambsuf{n-1}{p} a = b q \leq p\,.
\end{equation*}
Then $a \leq \pre(\ambsuf{n-1}{p}) \notin I$ and $b \leq \suf(\ambpre{n-1}{p}) \notin I$. So, we can apply \cref{amb_odd_overlap_lift} to obtain an $n$-ambiguity $\tilde{p}$ with $\tilde{p} \leq bq \leq p$. Hence, $\tilde{p} = p$ and $q = \ambpre{n-1}{p}$. 
\end{proof}

The following \namecref{sub_amb_overlap} recovers \cite[Lemma~4.7]{Redondo/Roman:2018a}. 

\begin{lemma} \label{sub_amb_overlap}
Let $n$ be an even integer and $p$ a path in $Q$. We consider the ordered set
\begin{equation*}
\set{q \in \Gamma_n}{q \leq p} = 
\{q_1, \ldots, q_N\}
\end{equation*}
where $p = b_i q_i a_i$ and $a_i \lneq a_{i+1}$ for any $1 \leq i \leq N-1$. Then $\ambsuf{n-1}{q_i} = \ambpre{n-1}{q_{i+1}}$ for all $1 \leq i \leq N-1$. 
\end{lemma}
\begin{proof} 
If $\ambsuf{n-1}{q_i} \leq q_{i+1}$ or $\ambpre{n-1}{q_{i+1}} \leq q_i$ the claim holds by \cref{sub_amb_even}. On the other hand, if $\ambsuf{n-1}{q_i} \not\leq q_{i+1}$ and $\ambpre{n-1}{q_{i+1}} \not\leq q_i$, then we have
\begin{equation*}
b (\ambsuf{n-1}{q_i}) = (\ambpre{n-1}{q_{i+1}}) a \quad \text{with } a,b \notin I\,.
\end{equation*}
By \cref{amb_odd_overlap_lift} there exists an $n$-ambiguity $q \leq b (\ambsuf{n-1}{q_i})$. By assumption, we have $q = q_i$ or $q = q_{i+1}$. This is a contradiction.
\end{proof}

In \cref{amb_odd_overlap_lift}, we utilized that, under certain conditions, we can extend an $n$-ambiguity to an $(n+1)$-ambiguity. For such an extension to exist, it is crucial that one can choose $u_{n+1} \notin I$. The following example illustrates that this is not always possible.

\begin{example} \label{example_cone}
We consider the quiver
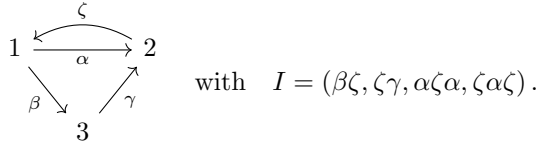
\begin{equation*}
\begin{tikzcd}[column sep=small]
1 \ar[rr,"\alpha" swap] \ar[dr,"{\beta}" swap] \&\& 2 \ar[ll,"\zeta" swap,bend right=30] \\
\& 3 \ar[ur,"\gamma" swap]
\end{tikzcd}
\quad \text{with} \quad I = (\beta\zeta,\zeta\gamma,\alpha\zeta\alpha,\zeta\alpha\zeta)\,.
\end{equation*}
Then, we have 2- and 4-ambiguities
\begin{equation*}
q = \zeta | \alpha\zeta | \alpha \leq \alpha | \zeta\alpha | \zeta | \alpha\zeta | \gamma = p,
\end{equation*}
respectively. If $r \leq \zeta\gamma$ such that $\alpha r \in I$, then $r = \zeta\gamma \in I$. Hence, $q$ cannot be extended by a prefix to a 3-ambiguity. 
\end{example}

We can now describe the Bardzell resolution.

\begin{theorem} \label{bardzell_resn}
Let $\kk$ be a field and $A$ a monomial algebra over $\kk$. The complex,
\begin{gather*}
\Bardzell(A)_{n+1} \colonequals \begin{cases}
0 & \text{if } n < -1 \\
A \otimes_E \kk \Gamma_n \otimes_E A & \text{if } n \geq -1
\end{cases}
\end{gather*}
with differential
\begin{equation} \label{bardzell_resn:diff}
\begin{gathered}
d_{n+1}(1 \otimes p \otimes 1) \colonequals \begin{cases}
\displaystyle\sum_{q \in \Sub(p)} \suf(q) \otimes q \otimes \pre(q) & \text{if } n \text{ odd} \\
\begin{aligned} &\suf(\ambpre{n-1}{p}) \otimes \ambpre{n-1}{p} \otimes 1 \\
&\quad - 1 \otimes \ambsuf{n-1}{p} \otimes \pre(\ambsuf{n-1}{p}) \end{aligned} & \text{if } n \text{ even}
\end{cases}
\end{gathered}
\end{equation}
for $p \in \Gamma_n$, is a resolution of $A$ over $\env{A}$ with quasi-isomorphism
\begin{equation*}
\epsilon \colon \Bardzell(A) \to A \,,\quad b \otimes p \otimes a \mapsto \begin{cases} ba & \text{if } p \in \Gamma_{-1} \text{ and } t(a) = p = s(b) \\ 0 & \text{else}\,. \end{cases}
\end{equation*}
\end{theorem}

We need a few more technical results for the proof. First, we describe some constraints on the position of a smaller ambiguity in a bigger one.

\begin{lemma} \label{amb_sub_position}
Let $n$ be an integer and $i+j=n-1$. We assume $p$ is an $n$-ambiguity with a decomposition $p = b q a$ such that $q$ is an $i$-ambiguity. 
\begin{enumerate}
\item If $i$ is odd, then $a = 1$ if and only if $\ambsuf{j}{p} \leq b$. 
\item If $i$ is even, then $a \lneq \pre(\ambsuf{n-1}{p})$ if and only if $\ambsuf{j}{p} \leq b$.
\end{enumerate}
The analogous statements for $\ambpre{j}{p}$ hold as well: 
\begin{enumerate}
\item If $i$ is odd, then $b = 1$ if and only if $\ambpre{j}{p} \leq a$.
\item If $i$ is even, then $b \lneq \suf(\ambpre{n-1}{p})$ if and only if $\ambpre{j}{p} \leq a$.
\end{enumerate}
\end{lemma}

\begin{proof}
We decompose $p$ and $q$ as left ambiguities $p = u_0 \ldots u_n$ and $q = u'_0 \ldots u'_i$, 
and let $\ell$ be maximal such that $u_\ell \ldots u_n \geq q$.

Let $i$ be odd. The maximality of $\ell$ implies that $u_0'\leq u_\ell$. Therefore, by \cref{amb_containment}, we have $u_{\ell+i} \leq u'_i$. Hence,
\begin{equation*}
a = 1 \iff \ell+i=n \iff u_{j+1} \ldots u_n \geq q \iff \ambsuf{j}{p} \leq b\,.
\end{equation*}

Let $i$ be even. Then $\pre(\ambsuf{n-1}{p}) = u_n$. By \cref{amb_containment} we have $u_{\ell+i} \geq u'_i$. Hence,
\begin{equation*}
a \lneq \pre(\ambsuf{n-1}{p}) \iff \ell+i=n \iff u_{j+1} \ldots u_n \geq q \iff \ambsuf{j}{p} \leq b\,. \qedhere
\end{equation*}
\end{proof}

\begin{lemma} \label{amb_even_pre_suf_zero}
Let $m < n$ be integers and $p$ an $n$-ambiguity. We assume $p \hat{a} = bqa$ with $q$ an $m$-ambiguity and $m$ even. Then, $b \notin I$ if and only if $q \leq \ambsuf{m+1}{p}$. In particular, $q \leq p$. 

An analogous statement holds for $\hat{b} p = bqa$ and $\ambpre{m+1}{p}$.
\end{lemma}
\begin{proof}
Let $p = u_0 \ldots u_n$ and $q = u'_0 \ldots u'_m$ be decompositions as left ambiguities. We assume $b \notin I$. Then, $b < u_0 u_1$. In particular, this means $u'_0 \leq u_0 u_1$. If $u'_0 = u_0$, then $q = \ambsuf{m}{p} \leq \ambsuf{m+1}{p}$. If $u'_0 \leq u_1$, then by \cref{amb_containment} we have $q \leq u_1 \ldots u_{m+1} \leq \ambsuf{m+1}{p}$ since $m$ is even.

The converse is clear from the definition of $\ambsuf{m+1}{p}$.
\end{proof}

\begin{proof}[Proof of \cref{bardzell_resn}]
In the definition of the differential, we use implicitly that for $n$ even an $n$-ambiguity contains exactly two $(n-1)$-ambiguities; see \cref{sub_amb_even}. First, we show that $(\Bardzell(A),d)$ is a complex, that is that $d^2 = 0$. 

\begin{step}
We show that, for $n$ odd, we have $d^2(1 \otimes p \otimes 1) = 0$ when $p \in \Gamma_n$.
\end{step}
Using the definition of $d$ in \cref{bardzell_resn:diff}, we can write
\begin{align*}
d^2(1 \otimes p \otimes 1) &= \sum_{q \in \Sub(p)} \left( \suf(q) \suf(\ambpre{n-2}{q}) \otimes \ambpre{n-2}{q} \otimes \pre(q) \right. \\
&\quad\quad \left.- \suf(q) \otimes \ambsuf{n-2}{q} \otimes \pre(\ambsuf{n-2}{q}) \pre(q)\right)\,.
\end{align*}

We order $\Sub(p) = \{\ambsuf{n-1}{p} = q_1, \ldots, q_N = \ambpre{n-1}{p}\}$ as in \cref{sub_amb_overlap}. Then $\ambsuf{n-2}{q_i} = \ambpre{n-2}{{q_{i+1}}}$. Thus we obtain
\begin{align*}
&\quad d^2(1 \otimes p \otimes 1) \\
&= \suf(\ambpre{n-1}{p}) \suf(\ambpre{n-2}{p}) \otimes \ambpre{n-2}{p} \otimes 1 - 1 \otimes \ambsuf{n-2}{p} \otimes \pre(\ambsuf{n-2}{p}) \pre(\ambsuf{n-1}{p}) \\
&= 0\,,
\end{align*}
since $\suf(\ambpre{n-1}{p}) \suf(\ambpre{n-2}{p}) \in I$ and analogously for the second term.

\begin{step}
We show that, for $n$ even, we have $d^2(1 \otimes p \otimes 1) = 0$ when $p \in \Gamma_n$.
\end{step}
We can write
\begin{align*}
d^2(1 \otimes p \otimes 1) &= \sum_{q \in \Sub(\ambpre{n-1}{p})} \suf(\ambpre{n-1}{p}) \suf(q) \otimes q \otimes \pre(q) \\
&\quad - \sum_{q \in \Sub(\ambsuf{n-1}{p})} \suf(q) \otimes q \otimes \pre(q) \pre(\ambsuf{n-1}{p})\,.
\end{align*}
For $q \in \Sub(\ambpre{n-1}{p}) \cap \Sub(\ambsuf{n-1}{p})$, the terms in the first and the second summand coincide and thus cancel. If $q \in \Sub(\ambpre{n-1}{p}) \setminus \Sub(\ambsuf{n-1}{p})$, then $q \not\leq \ambsuf{n-1}{p}$, and by \cref{amb_even_pre_suf_zero}, we have $\suf(\ambpre{n-1}{p}) \suf(q) \in I$. Analogously, the prefix of the second summand lies in $I$ when $q \in \Sub(\ambsuf{n-1}{p}) \setminus \Sub(\ambpre{n-1}{p})$. 

Thus, we have shown that $(\Bardzell(A),d)$ is a complex. It remains to show that $\epsilon$ is a quasi-isomorphism. We use the argument from \cite[Theorem~1]{Skoeldberg:2008}. 

\begin{step}
The map $\epsilon$ has an inverse up to homotopy as a map of right $A$-complexes. 
\end{step}
We consider the inclusion as right $A$-complexes
\begin{equation*}
\iota \colon A \to \Bardzell(A) \,,\quad a \mapsto 1 \otimes t(a) \otimes a
\end{equation*}
and the homotopy of right $A$-complexes
\begin{equation*}
\sigma \colon \Bardzell(A) \to \Bardzell(A) \,,\quad b \otimes p \otimes 1 \mapsto \sum_{\substack{bp=eqc\\q \in \Gamma_{n+1}}} e \otimes q \otimes c
\end{equation*}
when $p \in \Gamma_n$. By \cite[Theorem~1]{Skoeldberg:2008} one has
\begin{equation*}
d \sigma + \sigma d = \id + \iota \epsilon
\end{equation*}
as maps of right $A$-complexes.

Hence $\epsilon$ is a homotopy equivalence as a map of right $A$-complexes and thus, a quasi-isomorphism; the latter does not depend on the linearity. 
\end{proof}

\section{Diagonal map} \label{sec:diagonal}

The canonical resolution of $A$ as an $\env{A}$-module is the bar resolution $\barres(A)$; see \cite{Hochschild:1945} and also \cite[Section~IX.6]{Cartan/Eilenberg:1956}. Since both, Bardzell's resolution and the bar resolution, are complexes of free $\env{A}$-modules, there exist homotopy equivalences
\begin{equation} \label{bar_resolution:comparison}
\begin{tikzcd}
G \colon \barres(A) \ar[r,shift left] \& \Bardzell(A) \ar[l,shift left] \mskip6mu plus1mu \mathpunct{} \mkern-\thinmuskip{:}\mskip2mu F \nospacepunct{.}
\end{tikzcd}
\end{equation}
These \emph{comparison morphisms} were studied in \cite{Redondo/Roman:2018a}. It is well known that the bar resolution is equipped with a diagonal map 
\begin{equation*}
\Delta_{\barres} \colon \barres(A) \to \barres(A) \otimes_A \barres(A) \, ,
\end{equation*} 
see \cite[\S2.3]{Witherspoon:2019}. Hence, the comparison morphisms \cref{bar_resolution:comparison} induce a diagonal map on Bardzell's resolution. That is, a diagonal map $\Delta_{\Bardzell}$ is defined such that the top of the following diagram commutes.
\begin{equation}\label{induced_diagonal}
\begin{tikzcd}
\& \barres(A) \ar[r,"\Delta_{\barres}"] \& \barres(A) \otimes_A \barres(A) \ar[dr,"G \otimes G"] \\
\Bardzell(A) \ar[ur,"F"] \ar[dr,"\epsilon" swap,"\simeq"] \ar[rrr,"\Delta_{\Bardzell}"] \& \& \& \Bardzell(A) \otimes_A \Bardzell(A) \ar[dl,"\epsilon \otimes \epsilon","\simeq" swap] \&[-50pt]\\
\& A \& A \otimes_A A \ar[l,"\cong" swap,"\mu"] 
\end{tikzcd}
\end{equation}
Redondo and Roman described the induced diagonal map for quadratic string algebras; see \cite{Redondo/Roman:2018b}.

Alternatively, a diagonal map can be defined as a lift of the canonical isomorphism given by the multiplication ${\mu \colon A\otimes_A A \to A}$ on $A$; see \cite[\S2.3]{Witherspoon:2019}. By the Comparison Theorem \cite[Proposition~1.3.1]{Avramov:1998}, the lift is unique up to homotopy. Hence, a lift $\Delta$ coincides with $\Delta_{\Bardzell}$ up to homotopy. In the following, we construct a map $\Delta$ and show this map is a lift of the isomorphism $\mu^{-1}$.

\begin{theorem} \label{diagonal}
The $\env{A}$-linear homogeneous map $\Delta \colon \Bardzell(A) \to \Bardzell(A) \otimes_A \Bardzell(A)$, given by
\begin{equation} \label{eq:Delta}
\Delta_{n+1}(1 \otimes p \otimes 1) \colonequals \sum_{\substack{i+j=n-1\\i,j \geqslant -1}} \sum_{\substack{p = c q_2 b q_1 a\\q_1 \in \Gamma_i, q_2 \in \Gamma_j}} c \otimes q_2 \otimes b \otimes q_1 \otimes a
\end{equation}
for any $p \in \Gamma_n$, is, up to homotopy, the map induced from the diagonal map on the bar resolution.
\end{theorem}

Note that, implicitly, we are identifying 
\begin{equation*}
(A \otimes_E \kk \Gamma_j \otimes _E A) \otimes_A (A \otimes_E \kk \Gamma_i \otimes_E A) \cong A \otimes_E \kk \Gamma_j \otimes _E A \otimes_E \kk \Gamma_i \otimes_E A\,.
\end{equation*}

Before proving \cref{diagonal}, we show how, depending on the parity of $i$ and $j$, the decompositions in \cref{eq:Delta} can be simplified.

\begin{lemma} \label{amb_split}
Let $n$ be an integer and $i+j=n-1$. We assume $p$ is an $n$-ambiguity with 
\begin{equation*}
p = c q_2 b q_1 a \quad \text{for } q_1 \in \Gamma_i \text{ and } q_2 \in \Gamma_j\,.
\end{equation*}
If $i$ is odd, then $a=1$, and if $j$ is odd, then $c=1$. In particular, if $i$ and $j$ are odd, then \cref{amb_split_suf_pre} is the unique decomposition of $p$ into an $i$- and $j$-ambiguity.
\end{lemma}
\begin{proof}
This follows from \cref{amb_sub_position}.
\end{proof}

\begin{lemma} \label{split-amb-bound}
Let $n \geq i,j$ be integers and $p$ an $n$-ambiguity with
\begin{equation*}
p = c q_2 b q_1 a \quad\text{for } q_1 \in \Gamma_i, q_2 \in \Gamma_j\,.
\end{equation*}
Then, $i+j \leq n-1$. 
\end{lemma}
\begin{proof}
We take decompositions as left ambiguities $p = u_0 \ldots u_n$, $q_1 = u'_0 \ldots u'_i\,$,
and pick $\ell$ maximal such that $q_1 \leq u_\ell \ldots u_n$. Then $q_2 \lneq u_0 \ldots u_\ell$, and hence $j < \ell$. 

We use \cref{amb_containment}. If $i$ is even, then $q_1 \not\leq u_\ell \ldots u_{\ell+i-1}$ and hence $\ell+i \leq n$. If $i$ is odd, then $q_1 \geq u_{\ell+1} \ldots u_{\ell+i}$ and hence $\ell+i \leq n$. Combining the inequalities we obtain $j+i < j+ \ell \leq n$, and thus, $i+j \leq n-1$.
\end{proof}

\begin{proof}[Proof of \cref{diagonal}]
As discussed at the beginning of this section, it is enough to show that $\Delta$ is a chain map and a lift of the isomorphism $\mu^{-1} \colon A \to A\otimes_A A$. It is straightforward to see that $\Delta$ is a lift of $\mu^{-1}$ since $\mu (\epsilon \otimes \epsilon) \Delta = \epsilon$ holds.

It remains to show that $\Delta$ is a chain map. Since the differential on $\Bardzell(A) \otimes_A \Bardzell(A)$ is $\id \otimes d + d \otimes \id$, we need to show that $\Delta d= (\id \otimes d+d\otimes\id)\Delta$. 

In the cases when $i$ or $j$ are equal to $-1$, there is nothing to show. So, we assume that $i,j \geq 0$ for the remainder of the proof. We break the proof into two cases, one for $n$ even and one for $n$ odd.

\begin{case}
Let $n$ be even and $p \in \Gamma_n$. The first summand of $(\id \otimes d + d \otimes \id)\Delta$ yields
\begin{subequations}
\begin{align}
&\quad (\id \otimes d) (\Delta(1 \otimes p \otimes 1)) \notag \\
=& \sum_{\substack{k+\ell=n-2\\k \text{ odd}}} \sum_{\substack{p=c q_2 b q_1 a\\q_2 \in \Gamma_\ell, q_1 \in \Gamma_{k+1}}} c \otimes q_2 \otimes b \suf(\ambpre{k}{q_1}) \otimes \ambpre{k}{q_1} \otimes a \label{eq:even-d-Delta:even-second-suf} \\
&\quad- \sum_{\substack{k+\ell=n-2\\k \text{ odd}}} \sum_{\substack{p=c q_2 b q_1 a\\q_2 \in \Gamma_\ell, q_1 \in \Gamma_{k+1}}} c \otimes q_2 \otimes b \otimes \ambsuf{k}{q_1} \otimes \pre(\ambsuf{k}{q_1}) a \label{eq:even-d-Delta:even-second-pre} \\
&\quad- \sum_{\substack{k+\ell=n-2\\k \text{ even}}} \sum_{\substack{p=c q_2 b q_1 a\\q_2 \in \Gamma_\ell, q_1 \in \Gamma_{k+1}}} \sum_{r \in \Sub(q_1)} c \otimes q_2 \otimes b \suf(r) \otimes r \otimes \pre(r) a\,. \label{eq:even-d-Delta:odd-second}
\end{align}
\end{subequations}
The Koszul sign rule when applying $(\id \otimes d)$ yields $(-1)^{\ell+1}$ for each summand. The second summand of the differential yields
\begin{subequations}
\begin{align}
&\quad (d \otimes \id)(\Delta(1 \otimes p \otimes 1)) \notag \\
&= \sum_{\substack{k+\ell=n-2\\k \text{ even}}} \sum_{\substack{p=c q_2 b q_1 a\\q_2 \in \Gamma_{\ell+1}, q_1 \in \Gamma_k}} \sum_{r \in \Sub(q_2)} c \suf(r) \otimes r \otimes \pre(r) b \otimes q_1 \otimes a \label{eq:even-d-Delta:even-first} \\
&\quad+ \sum_{\substack{k+\ell=n-2\\k \text{ odd}}} \sum_{\substack{p=c q_2 b q_1 a\\q_2 \in \Gamma_{\ell+1}, q_1 \in \Gamma_k}} c \suf(\ambpre{\ell}{q_2}) \otimes \ambpre{\ell}{q_2} \otimes b \otimes q_1 \otimes a \label{eq:even-d-Delta:odd-first-suf} \\
&\quad- \sum_{\substack{k+\ell=n-2\\k \text{ odd}}} \sum_{\substack{p=c q_2 b q_1 a\\q_2 \in \Gamma_{\ell+1}, q_1 \in \Gamma_k}} c \otimes \ambsuf{\ell}{q_2} \otimes \pre(\ambsuf{\ell}{q_2}) b \otimes q_1 \otimes a \label{eq:even-d-Delta:odd-first-pre}\,.
\end{align}
\end{subequations}
By \cref{amb_split}, we have $c=1$ in \cref{eq:even-d-Delta:even-first,eq:even-d-Delta:even-second-suf,eq:even-d-Delta:even-second-pre} and $a=1$ in \cref{eq:even-d-Delta:odd-first-suf,eq:even-d-Delta:odd-first-pre,eq:even-d-Delta:odd-second}. 

\begin{step} \label{step:even-telescoping}
Combine \cref{eq:even-d-Delta:even-second-suf,eq:even-d-Delta:even-second-pre}.
\end{step}

We fix $k+\ell=n-2$ with $k$ and $\ell$ odd. Let $r$ be the path $p = \ambsuf{\ell}{p} r$. Then, we order the set of all $(k+1)$-ambiguities contained in $r$,
\begin{equation*}
\set{\tilde{r} \in \Gamma_{k+1}}{\tilde{r} \leq r} = \{r_1, \ldots, r_N\}\,,
\end{equation*}
as in \cref{sub_amb_overlap}. By construction, we have $r_1 = \ambpre{k+1}{p}$. Note that in \cref{eq:even-d-Delta:even-second-suf,eq:even-d-Delta:even-second-pre}, we have $q_2 = \ambsuf{\ell}{p}$, and $q_1$ takes on precisely the values $r_1, \dotsc, r_N$. By \cref{sub_amb_overlap}, we have that $\ambsuf{k}{r_i} = \ambpre{k}{r_{i+1}}$ for any $1 \leq i \leq N-1$, and therefore, the two sums cancel except for the terms with $\ambpre{k}{r_1}$ and $\ambsuf{k}{r_{N}}$.

\begin{claim}
$\ambsuf{k}{r_N} = \ambpre{k}{\left(\ambsuf{n-1}{p}\right)}$.
\end{claim}

\begin{figure}[h]
\begin{equation*}
\begin{tikzcd}[row sep=tiny]
\bullet \ar[r,"a_N"] \& \bullet \ar[rrrrrr,"r_N"] \&[-2em] \&[-2em] \&[-2em] \&[+5em] \&[-2em] \&[-2em] \bullet \ar[rr,"b_N"] \&[-2em] \& \bullet \\
\&\&\& \bullet \ar[rrrr,"{\ambsuf{k}{r_N}}"] \&\&\&\& \bullet \\[-.5em]
\bullet \ar[rrrr,"\pre(\ambsuf{n-1}{p})" swap] \&\&\&\& \bullet \ar[rrrr,"{\ambpre{k}{\left(\ambsuf{n-1}{p}\right)}}" swap] \&\&\&\& \bullet \ar[r,"\tilde{b}" swap] \& \bullet \\
\bullet \ar[rr,"\pre(\ambsuf{n-1}{p})" swap] \&\& \bullet \ar[rrrr,"{\ambpre{k}{\left(\ambsuf{n-1}{p}\right)}}" swap] \&\&\&\& \bullet \ar[rrr,"\tilde{b}" swap] \&\&\& \bullet
\end{tikzcd}
\end{equation*}
\caption{Graphical depiction of the decompositions of the path $r$, defined by $p = \ambsuf{\ell}{p} r$ in \cref{step:even-telescoping}. The two possible configurations are in the third and fourth line. It is shown that the only configuration that occurs has $\ambsuf{k}{r_N} = \ambpre{k}{\left(\ambsuf{n-1}{p}\right)}$.} \label{step:even-telescoping:setup}
\end{figure}
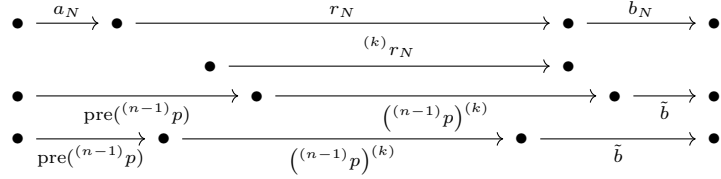

Since $k$ and $\ell$ are odd, there is a unique decomposition of $\ambsuf{n-1}{p}$ into a $k$- and a $\ell$-ambiguity. That is,
\begin{equation*}
\ambsuf{n-1}{p} = \ambsuf{\ell}{p} \tilde{b} \ambpre{k}{\left(\ambsuf{n-1}{p}\right)},
\end{equation*}
for some $\tilde{b}$; see \cref{amb_split}. Note that we must have $\ambpre{k}{\left(\ambsuf{n-1}{p}\right)} \leq r$. If the claim is false, then $\ambsuf{k}{r_N} \not\leq \ambsuf{n-1}{p}$. It is straightforward to check that we can apply \cref{amb_odd_overlap_lift} to $\ambsuf{k}{r_N}$ and $\ambpre{k}{\left(\ambsuf{n-1}{p}\right)}$. There are two configurations of these paths in $r$ as depicted in \cref{step:even-telescoping:setup}. So, there exists a $(k+1)$-ambiguity $\tilde{r} \leq r$ with $\ambsuf{k}{\tilde{r}} = \ambpre{k}{\left(\ambsuf{n-1}{p}\right)}$ or $\tilde{r} \leq \tilde{b} \ambpre{k}{\left(\ambsuf{n-1}{p}\right)}$. The first is a contradiction to the fact that $r_N$ is the maximal element in the set of $(k+1)$-ambiguities contained in $r$, and the second a contradiction to \cref{amb_split}. Hence, the claim has been shown. 

Using the above, we can rewrite \cref{eq:even-d-Delta:even-second-suf,eq:even-d-Delta:even-second-pre} as
\begin{equation*}
\sum_{\substack{k+\ell=n-2\\k \text{ odd}}} \left( 1 \otimes \ambsuf{\ell}{p} \otimes b_1 \otimes \ambpre{k}{p} \otimes 1 - 1 \otimes \ambsuf{\ell}{p} \otimes b_2 \otimes \ambpre{k}{\left(\ambsuf{n-1}{p}\right)} \otimes \pre(\ambsuf{n-1}{p}) \right)\,,
\end{equation*}
where $b_1$ and $b_2$ depend on $k$ and $\ell$ and are chosen appropriately. 

\begin{step}
Combine \cref{eq:even-d-Delta:even-second-suf,eq:even-d-Delta:even-second-pre,eq:even-d-Delta:odd-first-suf,eq:even-d-Delta:odd-first-pre}.
\end{step}

The same argument as in \cref{step:even-telescoping} simplifies \cref{eq:even-d-Delta:odd-first-suf,eq:even-d-Delta:odd-first-pre}. Combining the terms \cref{eq:even-d-Delta:even-second-suf,eq:even-d-Delta:even-second-pre,eq:even-d-Delta:odd-first-suf,eq:even-d-Delta:odd-first-pre} yields
\begin{equation} \label{eq:even-d-Delta:odd}
\begin{aligned}
&\quad \sum_{\substack{k+\ell=n-2\\k \text{ odd}}} \left(\suf(\ambpre{n-1}{p}) \otimes \ambsuf{\ell}{\!\left(\ambpre{n-1}{p}\right)} \otimes b' \otimes \ambpre{k}{p} \otimes 1 \right. \\
&\quad \left. - 1 \otimes \ambsuf{\ell}{p} \otimes b \otimes \ambpre{k}{\left(\ambsuf{n-1}{p}\right)} \otimes \pre(\ambsuf{n-1}{p})\right)\,,
\end{aligned}
\end{equation}
where $b$ and $b'$ depend on $k$ and $\ell$, and are chosen appropriately. 

\begin{step}
Combine \cref{eq:even-d-Delta:even-first,eq:even-d-Delta:odd-second}.
\end{step}

We fix $k+\ell=n-2$ with $k$ and $\ell$ even. Given $p = c q_2 b \ambpre{k+1}{p}$ with $q_2 \in \Gamma_\ell$ and $r \in \Sub(\ambpre{k+1}{p})$, we have $q_2 \leq \ambsuf{\ell+1}{p}$ by \cref{split-amb-bound}. By \cref{amb_sub_position} we obtain
\begin{equation*}
\pre(r) \lneq \pre(\ambsuf{n-1}{p}) \iff \ambsuf{\ell+1}p \leq c q_2 b \suf(r)\,.
\end{equation*}
This means, any summand of \cref{eq:even-d-Delta:odd-second} for which $\pre(r) \lneq \pre(\ambsuf{n-1}{p})$ holds appears as a summand of \cref{eq:even-d-Delta:even-first}. Thus, \cref{eq:even-d-Delta:even-first,eq:even-d-Delta:odd-second} together give
\begin{equation} \label{eq:even-d-Delta:even}
\begin{aligned}
&\quad \sum_{\substack{k+\ell=n-2\\k \text{ even}}} \sum_{\substack{\ambpre{n-1}{p}=c q_2 b q_1 a\\q_2 \in \Gamma_\ell, q_1 \in \Gamma_k}} \suf(\ambpre{n-1}{p}) c \otimes q_2 \otimes b \otimes q_1 \otimes a\\
&\quad -\sum_{\substack{k+\ell=n-2\\k \text{ even}}} \sum_{\substack{\ambsuf{n-1}{p}=c q_2 b q_1 a\\q_2 \in \Gamma_\ell, q_1 \in \Gamma_k}} c \otimes q_2 \otimes b \otimes q_1 \otimes a \pre(\ambsuf{n-1}{p})\,.
\end{aligned}
\end{equation}

Combining \cref{eq:even-d-Delta:odd,eq:even-d-Delta:even}, we obtain
\begin{align*}
&\quad d(\Delta(1 \otimes p \otimes 1)) \\
&= \sum_{k+\ell=n-2} \sum_{\substack{\ambpre{n-1}{p} = c q_2 b q_1 a\\q_2 \in \Gamma_\ell, q_1 \in \Gamma_k}} \suf(\ambpre{n-1}{p}) c \otimes q_2 \otimes b \otimes q_1 \otimes a \\
&\quad - \sum_{k+\ell=n-2} \sum_{\substack{\ambsuf{n-1}{p} = c q_2 b q_1 a\\q_2 \in \Gamma_\ell, q_1 \in \Gamma_k}} c \otimes q_2 \otimes b \otimes q_1 \otimes a \pre(\ambsuf{n-1}{p}) \\
&= \Delta(d(1 \otimes p \otimes 1))\,.
\end{align*}
\end{case}

\begin{case}
Let $n$ be an odd integer and $p$ an $n$-ambiguity. Then, the first summand of the differential on $\Bardzell(A) \otimes \Bardzell(A)$ yields
\begin{subequations}
\begin{align}
&\quad (\id \otimes d)(\Delta(1 \otimes p \otimes 1)) \notag \\
=& - \sum_{\substack{k+\ell=n-2\\k \text{ odd}}} \sum_{\substack{p = c q_2 b q_1 a\\q_2 \in \Gamma_\ell,q_1 \in \Gamma_{k+1}}} c \otimes q_2 \otimes b \suf(\ambpre{k}{q_1}) \otimes \ambpre{k}{q_1} \otimes a \label{eq:odd-d-Delta:even-second-suf} \\
&\quad + \sum_{\substack{k+\ell=n-2\\k \text{ odd}}} \sum_{\substack{p = c q_2 b q_1 a\\q_2 \in \Gamma_\ell, q_1 \in \Gamma_{k+1}}} c \otimes q_2 \otimes b \otimes \ambsuf{k}{q_1} \otimes \pre(\ambsuf{k}{q_1}) a \label{eq:odd-d-Delta:even-second-pre} \\
&\quad + \sum_{\substack{k+\ell=n-2\\k \text{ even}}} \sum_{\substack{p = c q_2 b q_1 a\\q_2 \in \Gamma_\ell, q_1 \in \Gamma_{k+1}}} \sum_{r \in \Sub(q_1)} c \otimes q_2 \otimes b \suf(r) \otimes r \otimes \pre(r) a\,. \label{eq:odd-d-Delta:odd-second}
\end{align}
\end{subequations}
The Koszul sign rule when applying $(\id \otimes d)$ yields $(-1)^{\ell+1}$ for each summand. The second summand of the differential yields
\begin{subequations}
\begin{align}
&\quad (d \otimes \id)(\Delta(1 \otimes p \otimes 1)) \notag \\
&= \sum_{\substack{k+\ell=n-2\\k \text{ even}}} \sum_{\substack{p = c q_2 b q_1 a\\q_2 \in \Gamma_{\ell+1}, q_1 \in \Gamma_k}} c \suf(\ambpre{\ell}{q_2}) \otimes \ambpre{\ell}{q_2} \otimes b \otimes q_1 \otimes a \label{eq:odd-d-Delta:even-first-suf} \\
&\quad - \sum_{\substack{k+\ell=n-2\\k \text{ even}}} \sum_{\substack{p = c q_2 b q_1 a\\q_2 \in \Gamma_{\ell+1}, q_1 \in \Gamma_k}} c \otimes \ambsuf{\ell}{q_2} \otimes \pre(\ambsuf{\ell}{q_2}) b \otimes q_1 \otimes a \label{eq:odd-d-Delta:even-first-pre} \\
&\quad + \sum_{\substack{k+\ell=n-2\\k \text{ odd}}} \sum_{\substack{p = c q_2 b q_1 a\\q_2 \in \Gamma_{\ell+1}, q_1 \in \Gamma_k}} \sum_{r \in \Sub(q_2)} c \suf(r) \otimes r \otimes \pre(r) b \otimes q_1 \otimes a\,. \label{eq:odd-d-Delta:odd-first}
\end{align}
\end{subequations}
By \cref{amb_split} we have $a=1$ and $c=1$ in \cref{eq:odd-d-Delta:odd-first,eq:odd-d-Delta:odd-second}. 

\begin{step}
Combine \cref{eq:odd-d-Delta:even-second-suf,eq:odd-d-Delta:even-second-pre}.
\end{step}

We fix $k+\ell=n-2$ with $k$ odd and $\ell$ even. For a decomposition $p = c q_2 r$ where $q_2 \in \Gamma_\ell$ such that $r$ contains a $(k+1)$-ambiguity, we order the set of all $(k+1)$-ambiguities contained in $r$,
\begin{equation*}
\set{\tilde{r} \in \Gamma_{k+1}}{\tilde{r} \leq r} = \{r_1, \ldots, r_N\}
\end{equation*}
as in \cref{sub_amb_overlap}. By construction, we have $r_1 = \ambpre{k+1}{p}$. So, we can rewrite \cref{eq:odd-d-Delta:even-second-suf,eq:odd-d-Delta:even-second-pre} as
\begin{equation} \label{eq:odd-d-Delta:telescoping}
\sum_{\substack{k+\ell=n-2\\k \text{ odd}}} \sum_{\substack{p=cq_2r\\q_2 \in \Gamma_\ell, r \geqslant \ambpre{k+1}{p}}} \left(c \otimes q_2 \otimes b_1 \otimes \ambsuf{k}{r_N} \otimes a_1 - c \otimes q_2 \otimes b_2 \otimes \ambpre{k}{p} \otimes 1 \right) \,,
\end{equation}
where $b_1$, $a_1$ and $b_2$ depend on $q_2$. 

\begin{claim}
There exists a unique $q \in \Sub(p)$ with $q_2 b_1 (\ambsuf{k}{r_N}) \leq q$. 
\end{claim}

Consider the decompositions as right ambiguities
\begin{equation*}
p = v_n \ldots v_0 \,, \quad q_2 = v'_\ell \ldots v'_0 \quad \text{and} \quad \ambsuf{k}{r_N} = \hat{v}_k \ldots \hat{v}_0\,.
\end{equation*}
If $v_1 v_0 \leq a_1$, then using \cref{split-amb-bound} twice yields $k+\ell+2 \leq n-1$. This is a contradiction. Hence, $v_1 v_0 < a_1$ and $\hat{v}_0 \leq v_1$. Since $k$ is odd, we have $\hat{v}_k \geq v_{k+1}$ by \cref{amb_containment}. Hence, there exists $\hat{v}_{k+1} \leq v_{k+2}$ such that $\hat{v}_{k+1} \ldots \hat{v}_0$ is an $(k+1)$-ambiguity. By the choice of $r_N$ we have $v'_0 \leq \hat{v}_{k+1} \leq v_{k+2}$. So, applying \cref{amb_containment} inductively, we can extend $\hat{v}_0 \ldots \hat{v}_{k+1}$ to a $k+\ell-1=n-1$-ambiguity that satisfies the desired conditions. Since $k$ is odd, the $(n-1)$-ambiguity $q$ is unique by \cref{amb_split}.

\begin{claim}
For any $q \in \Sub(p)$ with $q \neq \ambpre{n-1}{p}$, there exists a decomposition $p = c q_2 r$ with $q_2 \in \Gamma_\ell$ and $r \geq \ambpre{k+1}{p}$ such that $q_2 b (\ambsuf{k}{r_N}) \leq q$. 
\end{claim}

Set $q_2 = \ambsuf{\ell}{q}$ and write $p = c q_2 r$. By assumption $\suf(q) \lneq \suf(\ambpre{n-1}{p})$. Hence by \cref{amb_sub_position} we have $\ambpre{k+1}{p} \leq r$. By the previous claim there exists an $(n-1)$-ambiguity $q'$ satisfying the desired conditions, and by construction, $q = q'$. 

\begin{step}
Combine \cref{eq:odd-d-Delta:even-second-suf,eq:odd-d-Delta:even-second-pre,eq:odd-d-Delta:odd-first}.
\end{step}

Comparing \cref{eq:odd-d-Delta:odd-first} with the second summand of \cref{eq:odd-d-Delta:telescoping}, we see that the summand $\suf(r) \otimes r \otimes \pre(r) b \otimes \ambpre{k}{p} \otimes 1$ appears in \cref{eq:odd-d-Delta:telescoping} if and only if {$\pre(r) b \ambpre{k}{p} \geq \ambpre{k+1}{p}$}. By \cref{amb_sub_position}, this is equivalent to $\suf(r) \lneq \suf(\ambpre{n-1}{p})$. Hence, the terms cancel, and the remaining terms of \cref{eq:odd-d-Delta:odd-first} are precisely those terms where {$\suf(r) \geq \suf(\ambpre{n-1}{p})$}. The first summand of \cref{eq:odd-d-Delta:telescoping} yields the terms for any other $q \in \Sub(p)$ by the above claims. Thus, we can write \cref{eq:odd-d-Delta:even-second-suf,eq:odd-d-Delta:even-second-pre,eq:odd-d-Delta:odd-first} as
\begin{equation*}
\sum_{q \in \Sub(p)} \sum_{\substack{k+\ell=n-2 \\ k \text{ odd}}} \sum_{\substack{q = c q_2 b q_1 a\\q_2 \in \Gamma_\ell, q_1 \in \Gamma_k}} \suf(q) c \otimes q_2 \otimes b \otimes q_1 \otimes a \pre(q)\,.
\end{equation*}
Analogously, we can simplify \cref{eq:odd-d-Delta:even-first-suf,eq:odd-d-Delta:even-first-pre,eq:odd-d-Delta:odd-second}, and obtain
\begin{equation*}
d(\Delta(1 \otimes p \otimes 1)) = \Delta(d(1 \otimes p \otimes 1))\,. \qedhere
\end{equation*}
\end{case}
\end{proof}

\section{The cup product of Hochschild cohomology}\label{sec:cup-product}

Since $\Bardzell(A)$ is a free resolution of $A$ over $\env{A}$, we can use it to give an explicit description of Hochschild cohomology $\HH{A} \colonequals \Ext{*}{\env{A}}{A}{A}$. Using the Hom-tensor adjunction we obtain
\begin{equation*}
\Hom{\env{A}}{\Bardzell_{n+1}(A)}{A} = \Hom{\env{A}}{A \otimes_E \kk \Gamma_n \otimes_E A}{A} \cong \Hom{\env{E}}{\kk \Gamma_n}{A} \equalscolon \kk \parallel{\Gamma_n}{\cB}\,,
\end{equation*}
and we denote the complex with modules $\kk \parallel{\Gamma_n}{\cB}$ with the induced differential by $\kk \parallel{\Gamma}{\cB}$. A basis of $\kk \parallel{\Gamma_n}{\cB}$ is given by the maps
\begin{equation*}
(\parallel{p}{b})(q) = \begin{cases}
b & \text{if } q=p \\
0 & \text{else}
\end{cases} \quad \text{for } q \in \Gamma_n,
\end{equation*}
where $p \in \Gamma_n$ and $b \in \cB$ such that $s(p) = s(b)$ and $t(p) = t(b)$. The notation $\parallel{-}{-}$ indicates that the paths $p$ and $b$ are parallel, starting and ending at the same vertex. 

We give an explicit description of the differential. We need to treat the even and odd case separately. 

\begin{case}
Let $n$ be an even integer and $\parallel{p}{b} \in \kk \parallel{\Gamma_{n-1}}{\cB}$. For an $n$-ambiguity $q$, the differential is given by
\begin{equation} \label{del_even}
\partial^n(\parallel{p}{b})(q) = (\parallel{p}{b})(d_{n+1}(q)) = \begin{cases}
0 & \text{if } p \not\leq q \\
cb - ba & \text{if } cp = pa = q \\
cb & \text{if } cp = q \text{ and $p$ not a suffix of $q$} \\
- ba & \text{if } pa = q \text{ and $p$ not a prefix of $q$} \,.
\end{cases}
\end{equation}
\end{case}

\begin{case}
Let $n$ be an odd integer and $\parallel{p}{b} \in \kk \parallel{\Gamma_{n-1}}{\cB}$. For an $n$-ambiguity $q$, the differential is given by
\begin{equation} \label{del_odd}
\partial^n(\parallel{p}{b})(q) = \begin{cases}
0 & \text{if }p \not\leq q \\
\displaystyle \sum_{i=1}^N c_i b a_i & \text{if } c_i p a_i = q \text{ and } a_i \neq a_j \text{ for } i \neq j \,.
\end{cases}
\end{equation}
\end{case}

The diagonal map on $\Bardzell(A)$ given in \cref{diagonal} induces a chain map
\begin{equation*}
\smile \colon \Hom{\env{A}}{\Bardzell(A)}{A} \times \Hom{\env{A}}{\Bardzell(A)}{A} \to \Hom{\env{A}}{\Bardzell(A)}{A}\,,
\end{equation*}
which induces the cup product on Hochschild cohomology; see the discussion before \cref{diagonal}. The cup product is given by the formula
\begin{equation} \label{cup_prod_Hom}
(\parallel{p_2}{b_2} \smile \parallel{p_1}{b_1})(q) = \sum_{q = e p_2 c p_1 a} e b_2 c b_1 a
\end{equation}
for any $(i+j-1)$-ambiguity $q$ where $\parallel{p_1}{b_1} \in \kk \Gamma_i$ and $\parallel{p_2}{b_2} \in \kk \Gamma_j$.

By the discussion at the beginning of \cref{sec:diagonal}, preceding \cref{diagonal}, this cup product coincides with the classical cup product coming from the bar resolution up to homotopy. Hence, we can use \cref{del_even,del_odd} to compute Hochschild cohomology and \cref{cup_prod_Hom} to describe the algebra structure on the Hochschild cohomology. We utilize the formulas to describe the Hochschild cohomology algebra in the following examples.

\begin{example}
Let $A = \kk Q/I$ be a finite dimensional quadratic monomial algebra. Then, any $n$-ambiguity is of length $(n+1)$, and in the formula for the diagonal map \cref{eq:Delta}, there is at most one summand for every $i+j=n-1$ given by $p = q_2 q_1$. On $\kk \parallel{\Gamma}{\cB}$ this yields the formula
\begin{equation*}
\parallel{p_2}{b_2} \smile \parallel{p_1}{b_1} = \begin{cases}
\parallel{p_2 p_1}{b_2 b_1} & \text{if } p_2 p_1 \in \Gamma_{i+j-1} \text{ and } b_2 b_1 \in \cB \\
0 & \text{else}
\end{cases}
\end{equation*}
for any $\parallel{p_1}{b_1} \in \kk \Gamma_i$ and $\parallel{p_2}{b_2} \in \kk \Gamma_j$. This is the same cup product that Redondo and Roman obtain using the comparison maps; see \cite[4.2]{Redondo/Roman:2018b}. 
\end{example}

\begin{example}
We consider the same quiver with relations as in \cref{example_cone}:
\begin{equation*}
\begin{tikzcd}[column sep=small]
1 \ar[rr,"\alpha" swap] \ar[dr,"{\beta}" swap] \&\& 2 \ar[ll,"\zeta" swap,bend right=30] \\
\& 3 \ar[ur,"\gamma" swap]
\end{tikzcd}
\quad \text{with} \quad I = (\beta\zeta,\zeta\gamma,\alpha\zeta\alpha,\zeta\alpha\zeta)\,.
\end{equation*}
The monomial algebra $A = \kk Q/I$ is a string algebra, but not gentle. By direct calculation one obtains
\begin{align*}
\HH[0]{A} &= \langle \parallel{1}{1} + \parallel{2}{2} + \parallel{3}{3}, \parallel{1}{\zeta\alpha}, \parallel{2}{\alpha\zeta} \rangle \\
\HH[4m+1]{A} &= \langle \parallel{(\alpha\zeta)^{3m}\alpha}{\alpha}, \parallel{(\zeta\alpha)^{3m}\zeta}{\zeta}, \parallel{(\alpha\zeta)^{3m}\alpha}{\gamma \beta}, \rangle \\
\HH[4m+2]{A} &= \langle \parallel{(\alpha\zeta)^{3m+1}\alpha}{\alpha} + \parallel{(\zeta\alpha)^{3m+1}\zeta}{\zeta}, \parallel{(\alpha\zeta)^{3m+1}\alpha}{\gamma\beta} \rangle \\
\HH[4m+3]{A} &= \langle \parallel{\beta(\zeta\alpha)^{3m}\zeta\gamma}{3}, \parallel{(\zeta\alpha)^{3m+2}}{\zeta\alpha} \rangle \\
\HH[4m+4]{A} &= \langle \parallel{\beta(\zeta\alpha)^{3m+1}\zeta\gamma}{3}, \parallel{(\alpha\zeta)^{3m+3}}{\alpha\zeta}, \parallel{(\zeta\alpha)^{3m+3}}{\zeta\alpha} \rangle
\end{align*}
for $m\geq 0$. The element $\parallel{1}{1} + \parallel{2}{2} + \parallel{3}{3}$ is the unit of $\HH{A}$. We describe, up to commutativity, all nonzero multiplications in positive degrees. We set $w \colonequals (\parallel{(\alpha\zeta)^{3m+1}\alpha}{\alpha} + \parallel{(\zeta\alpha)^{3m+1}\zeta}{\zeta}) \in \HH[4m+2]{A}$. Then we obtain
\begin{equation*}
\begin{aligned}
\parallel{(\alpha\zeta)^{3\ell}\alpha}{\alpha} \smile w &= \parallel{(\alpha\zeta)^{3(\ell+m)+2}}{\alpha\zeta} \\
\parallel{(\zeta\alpha)^{3\ell}\zeta}{\zeta} \smile w &= \parallel{(\zeta\alpha)^{3(\ell+m)+2}}{\zeta\alpha} \\
(\parallel{(\alpha\zeta)^{3\ell+1}\alpha}{\alpha} + \parallel{(\zeta\alpha)^{3\ell+1}\zeta}{\zeta}) \smile w &= \parallel{(\alpha\zeta)^{3(\ell+m+1)}}{\alpha\zeta} + \parallel{(\zeta\alpha)^{3(\ell+m+1)}}{\zeta\alpha}\,.
\end{aligned}
\end{equation*}
In particular, we see that the algebra $\HH{A}$ is not finitely generated as a $\kk$-algebra. However, the algebra $\HH{A}/\mathcal{N}$, where $\mathcal{N}$ is the ideal of nilpotent elements, is finitely generated; cf.\@ \cite{Green/Snashall/Solberg:2006}. 
\end{example}

\begin{example}\label{truncated cycle}
Let $m,n$ be positive integers such that $\operatorname{char}(\kk)$ does not divide $m$. We consider the quiver
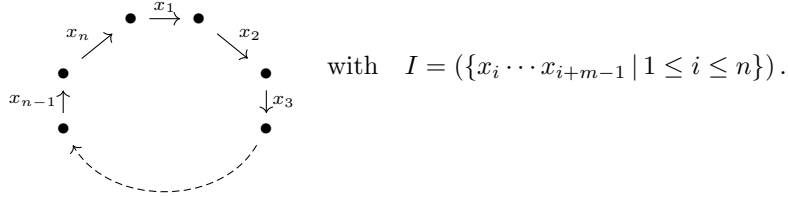
\begin{equation*}
\begin{tikzcd}[column sep=small,row sep=small]
\& \bullet \ar[r,"x_1"] \& \bullet \ar[dr,"x_2"] \\
\bullet \ar[ur,"x_n"] \& \& \& \bullet \ar[d,"x_3"] \\
\bullet \ar[u,"x_{n-1}"] \& \& \& \bullet \ar[lll,dashed,bend left=60]
\end{tikzcd}
\quad \text{with}\quad I = (\set{x_i \cdots x_{i+m-1}}{1 \leq i \leq n})\,.
\end{equation*}
The Hochschild cohomology of the associated monomial algebra $A$ was described in \cite{Locateli:1999}, and its algebra structure in \cite{Bardzell/Locateli/Marcos:2000}. The latter defined a product on $\kk \parallel{\Gamma_d}{\cB}$ and showed that, on cohomology, it coincides with the cup product. In the following, we will see, that our product on $\kk \parallel{\Gamma_d}{\cB}$, as described in \cref{cup_prod_Hom}, does not coincide with the product used in \cite{Bardzell/Locateli/Marcos:2000}. However, on cohomology, the products are the same.

The $(2\ell-1)$-ambiguities of $A$ are precisely the paths of length $\ell m$, and the $2\ell$-ambiguities the paths of length $\ell m+1$. For convenience, we will denote by $|p|$ the length of the path $p$. 

Let $\parallel{p}{b} \in \parallel{\Gamma_{i-1}}{\cB}$ and $\parallel{p'}{b'} \in \parallel{\Gamma_{j-1}}{\cB}$. Then 
\begin{equation*}
\parallel{p}{b} \smile \parallel{p'}{b'} = \parallel{pp'}{bb'} \quad \text{if $i$ or $j$ even}\,.
\end{equation*}
If $i$ and $j$ are odd, then $|p|+|p'|=m \frac{i+j}{2}-m+2$ and every $(i+j-1)$-ambiguity is of length $m \frac{i+j}{2}$. Hence,
\begin{equation*}
\parallel{p}{b} \smile \parallel{p'}{b'} = 0 \quad \text{if $i$ and $j$ odd and $|b|>0$ or $|b'|>0$}\,.
\end{equation*}
In the remaining case, when $|b| = |b'| = 0$, the cup product can have multiple terms. 

In comparison, the product on $\kk \parallel{\Gamma_d}{\cB}$ given in \cite{Bardzell/Locateli/Marcos:2000} is
\begin{equation*}
\parallel{p}{b} \mathrel{\overline{\smile}} \parallel{p'}{b'} = \begin{cases}
\parallel{pp'}{bb'} & \text{if } pp' \in \Gamma_{i+j-1} \\
0 & \text{else}\,.
\end{cases}
\end{equation*}
That is, the cases where the products need not coincide are for $i$ and $j$ odd and $|b|=|b'|=0$. 

By \cite[Lemma~6, Proposition~8]{Locateli:1999}, 
\begin{equation*}
\ker(\partial^{2 \ell+1}) = \bigoplus_{j=1}^{c-1} \kk \parallel{\Gamma_{2 \ell-1}}{Q_j}
\end{equation*}
where $Q_j$ the set of all paths of length $j$ in $Q$. In particular, elements $\parallel{p}{b} \in \parallel{\Gamma_{i-1}}{\cB}$ for $i$ odd and $|b|=0$ do not appear as nonzero summands in a cocycle. Hence, our product $\smile$ and the product $\mathrel{\overline{\smile}}$ of \cite{Bardzell/Locateli/Marcos:2000} coincide on the cocycles and therefore on cohomology. 
\end{example}

\section{Hochschild Cohomology of Triangular Monomial Algebras}

A quiver algebra with relations is \emph{triangular} if the underlying quiver has no oriented cycles. For a triangular monomial algebra $A$, the differentials \cref{del_even,del_odd} on the complex $\kk \parallel{\Gamma}{\cB}$ can be simplified. When $n$ is odd, the case $cp = pa = q$ cannot occur. When $n$ is even, the decomposition $cpa = q$ is unique, and hence, there is exactly one summand whenever $p \leq q$. These simplifications give us some control over a generating set of the cocycles given by the irreducible cocycles.

\begin{definition} \label{cocycle}
We fix a cocycle
\begin{equation} \label{cocycle_exp}
x = \sum_{i = 1}^k \alpha_i (\parallel{p_i}{b_i}) \in \kk \parallel{\Gamma_n}{\cB}
\end{equation}
of degree $n+1$ with $\alpha_i \neq 0$ for $1 \leq i \leq k$ and $\parallel{p_i}{b_i}$ pairwise distinct. We say $x$ is \emph{irreducible} if for each sequence $1 \leq i_1 < \ldots < i_\ell \leq k$ with $\ell < k$ and non-zero elements $\beta_j \in \kk $ for $1 \leq j \leq \ell$ one has
\begin{equation*} 
\partial\left(\sum_{j=1}^\ell \beta_j (\parallel{p_{i_j}}{b_{i_j}}) \right) \neq 0\,.
\end{equation*}
The irreducible elements form a generating set of the cocycles in $\kk \parallel{\Gamma_n}{\cB}$.
\end{definition}

\begin{lemma}\label{inter}
Given a triangular monomial algebra $A=\kk Q/I$ with basis $\cB$ and an irreducible cocycle $x=\sum_{i=1}^{k} \alpha_i (\parallel{p_i}{b_i}) \in \kk \parallel{\Gamma_{n}}{\cB}$ with $\alpha_i \in \kk$ for $n \geq 0$, there exist nontrivial paths $\tilde{p}$ and $\tilde{b}$ with $p_i = c_i \tilde{p} a_i$ and $b_i = c_i \tilde{b} a_i$ for $a_i, c_i \in \cB$ for all $1 \leq i \leq k$.
\end{lemma}

Before we prove the \namecref{inter}, we make an observation for general monomial algebras on overlapping elements of $\parallel{\Gamma_n}{\cB}$ for $n \geq 0$. Let $\partial(\parallel{p}{b})(q) = \partial(\parallel{p'}{b'})(q) \neq 0$ for $\parallel{p}{b}, \parallel{p'}{b'} \in \parallel{\Gamma_n}{\cB}$ and an $(n+1)$-ambiguity $q$. We choose $q = cpa = c'p'a'$ such that $cba = c'b'a'$. Without loss of generality, we may assume $a \leq a'$. Then, there exist nontrivial paths $\tilde{p} \leq p,p'$ and $\tilde{b} \leq b,b'$ such that
\begin{equation*}
p = \tilde{c} \tilde{p} \,,\quad p' = \tilde{p} \tilde{a} \,,\quad b = \tilde{c} \tilde{b} \quad \text{and} \quad b' = \tilde{b} \tilde{a}
\end{equation*}
for some $\tilde{a} \leq a'$ and $\tilde{c} \leq c$. This is depicted in \cref{fig:del_description}. Loosely speaking, this means that $b$ differs from $p$ at most along the `intersection' of $p$ and $p'$. 

\begin{figure}[h]
\begin{tikzpicture}
\coordinate (A) at (0,-0.2);
\coordinate (B) at ($(A)+(1,0)$);
\coordinate (C) at ($(B)+(1,0)$);
\coordinate (D) at ($(C)+(3,0)$);
\coordinate (E) at ($(D)+(1,0)$);
\coordinate (F) at ($(E)+(1,0)$);

\draw[->] (A) -- node[description] {$q$} (F);

\draw[->] ([yshift=1em,xshift=0.2em]A) -- node[description] {$a$} ([yshift=1em,xshift=-0.2em]B);
\draw[->] ([yshift=1em,xshift=0.2em]B) -- node[description] {$p$} ([yshift=1em,xshift=-0.2em]D);
\draw[->] ([yshift=1em,xshift=0.2em]D) -- node[description] {$c$} ([yshift=1em,xshift=-0.2em]F);

\draw[->] ([yshift=-1em,xshift=0.2em]A) -- node[description] {$a'$} ([yshift=-1em,xshift=-0.2em]C);
\draw[->] ([yshift=-1em,xshift=0.2em]C) -- node[description] {$p'$} ([yshift=-1em,xshift=-0.2em]E);
\draw[->] ([yshift=-1em,xshift=0.2em]E) -- node[description] {$c'$} ([yshift=-1em,xshift=-0.2em]F);

\draw[->] ([yshift=2em,xshift=0.2em]B) -- ([yshift=2em,xshift=0.2em]C) to[bend left=30] node[description] {$b$} ([yshift=2em,xshift=-0.2em]D);

\draw[->] ([yshift=-2em,xshift=0.2em]C) to[bend right=30] node[description] {$b'$} ([yshift=-2em,xshift=-0.2em]D) -- ([yshift=-2em,xshift=-0.2em]E);
\end{tikzpicture}
\caption{Graphical representation of the arrangement of the paths when $\partial(\parallel{p}{b})(q) = \partial(\parallel{p'}{b'})(q) \neq 0$ as described in \cref{del_odd}. All the parallel arrows lie on the same path, same for the bend arrows.} \label{fig:del_description}
\end{figure}
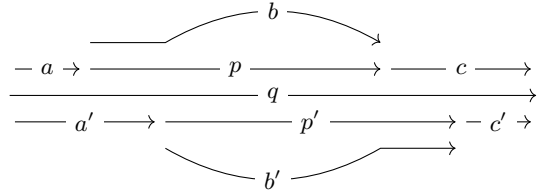

\begin{proof}[Proof of \cref{inter}]
For $n=0$ it is enough to show that any irreducible cocycle has at most one non-zero summand. Suppose $\partial(\parallel{p}{b})(q) = \partial(\parallel{p'}{b'})(q) \neq 0$ with $p \neq p'$. Since $A$ is triangular, $b=p$, $b'=p'$ and $\partial(\parallel{p}{p})(q) = \partial(\parallel{p'}{p'})(q) = q = 0$ as $Q \in \Gamma_1$ is a relation in $A$. This is a contradiction. Hence any irreducible cocycle has at most one non-zero summand and the claim is satisfied.

Suppose $n \geq 1$ and $k > 1$. Given a sequence $i_0, \ldots, i_\ell$ of positive integers $\leq k$ and $n$-ambiguities $q_1, \ldots ,q_\ell$ such that:
\begin{gather*}
\partial(\parallel{p_{i_j}}{b_{i_j}})(q_j) = c_j b_{i_j} a_j = c'_{j+1} b_{i_{j+1}} a'_{j+1} = \partial(\parallel{p_{i_{j+1}}}{b_{i_{j+1}}})(q_j) \neq 0
\end{gather*}
for all $1 \leq j \leq \ell$, we will prove by induction on $\ell$ that there exist nontrivial paths $\tilde{p}$ and $\tilde{b}$ with $p_{i_j} = \tilde{a}_j \tilde{p} \tilde{c}_j$ and $b_{i_j} = \tilde{a}_j \tilde{b} \tilde{c}_j$ with $\tilde{a}_j, \tilde{c}_j \in \cB$ for all $1 \leq j \leq \ell$. 

For $\ell = 1$ this holds by the previous discussion. Suppose the claim holds for $\ell$, and we want to prove it for $\ell+1$. Then
\begin{equation*}
\partial(\parallel{p_{i_{\ell}}}{b_{i_{\ell}}})(q_{\ell}) = c_{\ell} \tilde{c}_{\ell} \tilde{b} \tilde{a}_{\ell} a_{\ell} = c'_{\ell+1} b_{i_{\ell+1}} a'_{\ell+1} \neq 0
\end{equation*}
and $q = c_{\ell} \tilde{c}_{\ell} \tilde{p} \tilde{a}_{\ell} a_{\ell} = c'_{\ell+1} p_{i_{\ell+1}} a'_{\ell+1}$. Hence, there exists a path $\tilde{p}'$ such that $\tilde{p} = \tilde{c} \tilde{p}' \tilde{a}$ and $p_{i_{\ell+1}} = \tilde{c}'_{\ell+1} \tilde{p}' \tilde{a}'_{\ell+1}$ with $\tilde{a}, \tilde{c}, \tilde{c}'_{\ell+1}, \tilde{a}'_{\ell+1} \in \cB$. The latter holds since $c_{\ell} \tilde{c}_{\ell}, \tilde{a}_{\ell} a_{\ell}, c'_{\ell+1}, a'_{\ell+1} \in \cB$. One can think of $\tilde{p}'$ as the intersection of $\tilde{p}$ and $p_{i_{\ell+1}}$. We let $\tilde{b}'$ be the path parallel to $\tilde{p}'$ inside $\tilde{b}$. Since $A$ is triangular, the path $\tilde{b}'$ is unique. It remains to set $\tilde{a}'_j \colonequals \tilde{a}_j \tilde{a}$ and $\tilde{c}'_j \colonequals \tilde{c}_j \tilde{c}$ for $1 \leq j \leq \ell$. These elements together with $\tilde{p}'$ and $\tilde{b}'$ satisfy the desired conditions.

Since $x$ is irreducible, there exists a sequence from $\parallel{p_1}{b_1}$ to $\parallel{p_k}{b_k}$ passing through every $\parallel{p_i}{b_i}$ with $i=2, \cdots, k-1$ satisfying the conditions above. This finishes the proof. 
\end{proof}

\begin{lemma} \label{tria_cup_Hom}
Let $A = \kk Q/I$ be a triangular monomial algebra and $x \in \kk \parallel{\Gamma_{m-1}}{\cB}$, $y \in \kk \parallel{\Gamma_{n-1}}{\cB}$ irreducible cocycles with $n,m>0$. If $x \smile y \neq 0$, then $y \smile x = 0$. 
\end{lemma}
\begin{proof}
We take $\tilde{p}$ and $\tilde{p}'$ for $x$ and $y$ as in \cref{inter}. Since $x \smile y \neq 0$, there exists $q \in \Gamma_{n+m-1}$ such that $q = e \tilde{p} c \tilde{p}' a$. Let us suppose that $y\smile x\neq 0$. Then, there exists $q' \in \Gamma_{n+m-1}$ such that $q' = e' \tilde{p}' c' \tilde{p} a$. Hence, $\tilde{p} c \tilde{p}' c'$ is an oriented cycle in $Q$. This is a contradiction since $A$ is a triangular monomial algebra.
\end{proof}

\begin{theorem}\label{main}
Consider a triangular monomial algebra, $A=\kk Q/I$, with a basis of paths $\cB$. Then, the cup product on $\HH{A}$ is zero in positive degrees.
\end{theorem}
\begin{proof}
It is enough to show that the equivalence class of a cup product of irreducible cocycles is zero. Let $x \in \kk \parallel{\Gamma_{n-1}}{\cB}$ and $y \in \kk \parallel{\Gamma_{m-1}}{\cB}$ be irreducible cocycles. By \cref{tria_cup_Hom}, we have $x \smile y = 0$ or $y \smile x = 0$. Since $\smile$ is graded-commutative on $\HH[n+m]{A}$, we conclude that $\overline{x\smile y}$ is zero in $\HH[m+n]{A}$.
\end{proof}

Finally, we give an example to demonstrate that there are nontrivial examples of \cref{tria_cup_Hom}. That is, there exist irreducible cocycles $x$ and $y$ with $x \smile y \neq 0$ although $\overline{x \smile y} = 0$. 

\begin{example}
We consider the quiver 
\begin{equation*}
Q = 
\begin{tikzcd}
\bullet \ar[r,"\alpha_1" swap] \& \bullet \ar[r,"\alpha_2" swap] \ar[r,bend left,"\beta"] \& \bullet \ar[r,"\alpha_3" swap] \& \bullet \ar[r,"\alpha_4" swap] \ar[r,bend left,"\gamma"] \& \bullet \ar[r,"\alpha_5" swap] \& \bullet
\end{tikzcd}
\end{equation*}
with $I = (\alpha_5 \alpha_4, \alpha_4 \alpha_3, \alpha_3 \alpha_2, \alpha_2 \alpha_1)$. It is straightforward to check that
\begin{equation*}
x \colonequals \parallel{\alpha_4 \alpha_3}{\gamma \alpha_3} + \parallel{\alpha_5 \alpha_4}{\alpha_5 \gamma} \quad \text{and} \quad y \colonequals \parallel{\alpha_2 \alpha_1}{\beta \alpha_1} + \parallel{\alpha_3 \alpha_2}{\alpha_3 \beta}
\end{equation*}
are irreducible cocycles. Then
\begin{equation*}
y \smile x = \parallel{\alpha_4 \alpha_3 \alpha_2 \alpha_1}{\gamma \alpha_3 \beta \alpha_1} + \parallel{\alpha_5 \alpha_4 \alpha_3 \alpha_2}{\alpha_5 \gamma \alpha_3 \beta} = \partial(\parallel{\alpha_4 \alpha_3 \alpha_2}{\gamma \alpha_3 \beta}) \neq 0\,.
\end{equation*}
\end{example}

\bibliographystyle{amsalpha}
\bibliography{./References}

\end{document}